\documentclass[12pt,oneside,leqno]{amsart}
\usepackage[margin=1in]{geometry}
\usepackage{amsmath}
\usepackage{amsthm}
\usepackage{amsfonts}
\usepackage{amssymb}
\usepackage{stmaryrd}
\usepackage[hidelinks]{hyperref}
\usepackage{tikz}
\usetikzlibrary{cd}
\usepackage{mathrsfs}

\makeatletter
\renewcommand{\eqref}[1]{\textup{(\ignorespaces\ref{#1}\unskip\@@italiccorr)}}
\def\maketag@@@#1{\hbox{\m@th\normalfont\bfseries#1}}
\makeatother

\numberwithin{equation}{section}
\numberwithin{figure}{section}

\theoremstyle{plain}
\newtheorem{thm}[equation]{Theorem}
\newtheorem{lemma}[equation]{Lemma}

\newtheorem{prop}[equation]{Proposition}

\theoremstyle{definition}

\newtheorem{example}[equation]{Example}

\newcommand{\Q}{\ensuremath \mathbb{Q}}
\newcommand{\R}{\ensuremath \mathbb{R}}
\newcommand{\C}{\ensuremath \mathbb{C}}
\newcommand{\Z}{\ensuremath \mathbb{Z}}

\usepackage[foot]{amsaddr}

\begin{document}

\title[An elementary algebraic proof of the fundamental theorem of algebra]{An elementary algebraic proof of the fundamental theorem of algebra}

\author{Katelyn S.\ Clark$^{1}$}
\address{$^{1}$Department of Mathematics, Brigham Young University, Provo, UT 84602, USA}
\author{Pace P.\ Nielsen{$^{1,\ast}$}}
\address{$^{\ast}$ Corresponding author: \rm{pace@math.byu.edu}}


\begin{abstract}
We provide a new, elementary, algebraic proof of the fundamental theorem of algebra.  Furthermore, the method recovers a more general version of the theorem recently obtained by Joseph Shipman.  The key idea is to understand extension fields in which a polynomial gains a factor of a given degree.
\end{abstract}

\begingroup
\def\uppercasenonmath#1{} 
\let\MakeUppercase\relax 
\maketitle
\endgroup

\section{Introduction and history}

The fundamental theorem of algebra (FTA) asserts that $\C$ is an algebraic closure of $\R$ (and the only one, up to isomorphism).  There are many proofs of the FTA.  These proofs reveal interesting information about $\C$, and they employ tools from various mathematical fields ranging from topology and complex analysis to geometry and dynamical systems.

There are also algebraic proofs of the FTA.  One is due to Gauss, published in 1816 (see pages 33--56 of \cite{Gauss}).  This was Gauss's second proof of the FTA, but it is thought to be the very first proof of the FTA without any holes at the time of writing.  It proceeds by a clever induction over $2$-adic valuations.  Another algebraic proof is attributed to Emil Artin (see pages 615--617 of \cite{DF}).  It utilizes Galois theory and the existence of Sylow 2-subgroups.

In 2007 Joseph Shipman extracted the algebraic essence of the FTA and proved a more general statement about arbitrary algebraically closed fields (see \cite{Shipman}).  His proof relies on Galois theory and Sylow theorem combinatorics.

We provide a new algebraic proof of the FTA, with the method almost entirely belonging to the realm of field theory.  It was discovered independently of other algebraic proofs; however, one can find hints of those arguments in its structure.  It is elementary, requiring readers to understand only the part of ring theory used to force polynomials to have factors via long division, together with dimension counting arguments.  The proof recaptures Shipman's generalized FTA without the need for group theory.  It also introduces new facts about polynomial factorization over arbitrary fields.

\section{Squares and inequalities}\label{Section:2}

In this section we prove two well-known statements about $\R$ and $\C$, given as Lemma \ref{Lemma:QuadOverC} and Proposition \ref{Prop:OddDegree}.  Along the way, we make three new contributions, which are
\begin{itemize}
\item a construction of $\R$ from basic principles that should be more well-known, and which is simpler---in multiple ways---from many standard constructions,
\item a thorough explanation of how inequalities, positivity, and continuity over $\R$ are all encoded algebraic notions, leading to
\item a rebuttal of the common folklore that the FTA is a misnomer.
\end{itemize}
Those uninterested in the algebraization of those two well-known statements, or who are more interested in the new field-theoretic results, can safely skip to Section \ref{Section:Overview}.

We begin with a standard result about $\C$.

\begin{lemma}\label{Lemma:QuadOverC}
The field $\C$ is closed under square-roots, and thus has no quadratic extension.
\end{lemma}
\begin{proof}
Consider an arbitrary element $a+bi\in \C$, for some $a,b\in \R$.  If $b=0$, then the square-roots obviously lie in $\C$.  In the other case, the quantity $y:=(a+\sqrt{a^2+b^2})/2$ is positive and thus has a positive square-root in $\R$.  Taking $c:=\sqrt{y}$ and $d:=b/2c$, we quickly verify that $\pm (c+di)\in \C$ are the two square-roots of $a+bi$.

The second half of the lemma follows immediately from the quadratic formula.
\end{proof}

A key fact used in the previous proof is that the positive elements of $\R$ have square-roots in $\R$, which is an easy consequence of the intermediate value theorem (IVT).  Let us reprove the lemma without appealing to the IVT, instead using notions that are clearly algebraic.  It will be important to show that the usual ordering relation on $\R$ (and on some substructures) is expressible the language of algebra.

First, consider the (additive) semigroup freely generated by $\{1\}$, which we will denote by
\[
\Z_{+}:=\{1,2,3,\ldots\}.
\]
This semigroup is inductive.  Define multiplication on $\Z_{+}$ via iterated addition.  The usual ordering relation $<$ on $\Z_{+}$ is definable in the first-order language of semigroups; indeed, given $a,b\in \Z_{+}$, then $a< b$ is merely shorthand for the algebraic statement
\begin{equation}\label{Eq:Z+Ord}
\exists c:\ a+c=b.
\end{equation}
In this case, note that $b^2=(a+c)^2=a^2+(2ac+c^2)$; therefore, squaring is a monotone function on $\Z_{+}$.

Next, define $\Q_{+}$ in the usual way as the set of fractions over $\Z_{+}$ (where fractions are certain equivalence classes of ordered pair from $\Z_{+}$).  Addition and multiplication extend to $\Q_{+}$ in the usual way, as does the ordering relation.  Further, any inequalities over $\Q_{+}$ can be translated into the internal language of $(\Q_{+},+)$, again via \eqref{Eq:Z+Ord}.  The set $\Q_{+}$ also has a unary operation of (multiplicative) inversion, defined by flipping fractions.

There are many (isomorphic) ways to define $\R_{+}$ from $\Q_{+}$.  One of the easiest is to take the point of view that a positive real number is determined by the collection of strictly smaller positive rational numbers.  Thus, for example, $\sqrt{2}$ is determined by the initial segment
\[
a=\{x\in \Q_{+}\, :\, x<\sqrt{2}\}.
\]
Of course, to avoid circularity when actually defining $\sqrt{2}$, one should rewrite this initial segment as $a=\{x\in \Q_{+}\, :\, x^2<2\}$.

Guided by this viewpoint, define the set $\R_{+}$ as the collection of all initial segments of $\Q_{+}$ that are nonempty, proper, and have no greatest element.  In other words, define $a$ to be a positive real number if, and only if,
\[
a\in \mathscr{P}(\Q_+)\setminus\{\emptyset,\Q_+\}
\]
(where $\mathscr{P}$ is the powerset operation) with $a$ having no greatest element and with $a$ satisfying the initial segment property
\[
\forall x,y\in \Q_{+}:\ \text{if $x<y$ and $y\in a$, then $x\in a$}.
\]
We will call these initial segments (left) \emph{Dedekind cuts}.  The operations of addition, multiplication, and inversion all naturally extend to $\R_{+}$; if $a,b\in \R_{+}$, then
\begin{eqnarray*}
a+ b & := &\{x+y\in \Q_{+}\, :\, x\in a\text{ and }y\in b\},\\
a\cdot b & := & \{x\cdot y\in \Q_{+}\, :\, x\in a\text{ and }y\in b\}, \ \text{ and}\\
a^{-1} & := & \{x\in \Q_{+}\, :\, x^{-1}\notin a\}\ \ \text{(but removing its greatest element when it exists).}
\end{eqnarray*}
The definitions of the algebraic operations on $\R_{+}$ are significantly simpler than those for $\R$.  Without the presence of zero and negatives, many cases are avoided (which also simplifies proofs of their properties). The ordering relation on $\Q_{+}$ extends as the proper subset relation on Dedekind cuts, which orders $\R_{+}$.  Moreover, this order relation can be expressed in the language of $(\R_{+},+)$ by translating via \eqref{Eq:Z+Ord}.  Finally, the usual ordered ring axioms hold (except for those rules involving negation or zero); proofs are left to the interested reader.

We should point out that cancellativity of addition in $\R_{+}$ boils down to knowing that $\Q_{+}$ is archimedean under $<$, which is a second-order property.  Not every algebraic fact can be handled in a first-order way.

We are now ready to prove:

\begin{prop}\label{Prop:PositiveSquares}
Each element of $\R_{+}$ is a square.
\end{prop}
\begin{proof}
Let $a\in \R_{+}$ be arbitrary, and let
\[
b:=\{x\in \Q_{+}\, :\, x^2\in a\}.
\]
Showing that $b$ is a Dedekind cut is straightforward, since squaring is monotone by \eqref{Eq:Z+Ord}.  The inclusion $b^2\subseteq a$ is immediate.  It thus suffices to prove the reverse inclusion.

Let $q\in a$ be arbitrary.  Since $a$ has no greatest element, fix some $z\in \Q_{+}$, depending on $q$, such that $q+z\in a$.  Next, since $b$ is a proper initial segment, from the archimedean property on $\Q_{+}$ we may fix some $n\in \Z_{+}$ with $n\notin b$.  Taking $m\in \Z_{+}$ to be sufficiently large, we can guarantee that $3n/z < m$, again by the archimedean property.

A third use of the archimedean property tells us that there is some $x\in b$ with $x+1/m\notin b$.  Note that
\[
\left(x+\frac{1}{m}\right)^2=x^2+\frac{2mx+1}{m^2}<x^2+\frac{2mn+1}{m^2}\leq x^2+\frac{3mn}{m^2}<x^2+z.
\]
As $(x+1/m)^2\notin a$, we see that $x^2+z\notin a$, and hence $x^2+z>q+z$.  Consequently, $x^2>q$, and so, since $b^2$ is an initial segment of $\R_{+}$, we have $q\in b^2$, as desired.
\end{proof}

The proof of Proposition \ref{Prop:PositiveSquares} is algebraic, inasmuch as only standard concepts used in algebra appear in the proof, at least after rewriting all inequalities in terms of addition via \eqref{Eq:Z+Ord}.  The proof is \emph{additionally} analytic and topological.  The uses of the archimedean property are just an equivalent alternative to standard epsilonic continuity arguments.

Next, the field $\R$ is created from $\R_{+}$ via the usual Grothendieck group construction.  The algebraic operations on $\R_{+}$ extend to $\R$ in a piecewise manner, except that inversion is not defined on $0$.  The order relation extends as well.  It is impossible to express the order relation on $(\R,+)$ in the language of semigroups, but in light of Proposition \ref{Prop:PositiveSquares} we can express the order relation in the language of rings by saying that $a<b$ holds exactly when
\begin{equation}\label{Eq:RingOrder}
(a\neq b) \land (\exists c:\ a+c^2=b).
\end{equation}
Similar remarks holds for $\Z$ and $\Q$, after replacing the single existential variable in \eqref{Eq:RingOrder} with four existential variables, and replacing the square with a sum of four squares, by Lagrange's theorem.  Thus, while ostensibly working in the language of \emph{relational} algebra, all inequalities are disposable shorthands for simple algebraic statements (in each of the standard structures we care about).  In particular, continuity is an algebraic condition.

We will need one more important fact about $\R$, namely:

\begin{prop}\label{Prop:OddDegree}
Each real polynomial with odd degree has a real root.
\end{prop}
\begin{proof}
Let $f(x)\in \R[x]$ be a polynomial of odd degree $m\in \Z_{+}$.  Reduce to the case that $f(x)$ is monic and its constant term is nonzero (else zero is a root).  After replacing $f(x)$ by $-f(-x)$ if necessary, we may further assume that the constant term is negative.  (This is the only place in which we use the hypothesis that the degree is odd.)

Write $f(x)=f_2(x)-f_1(x)$, with $f_1(x),f_2(x)\in \R_{+}[x]$.  We claim that the set
\[
a:=\{q\in \Q_{+}\, :\, f_1(p)>f_2(p) \text{ for all $p\in \Q_{+}$ with $p\leq q$}\}
\]
will solve the equality $f(a)=0$, or equivalently $f_1(a)=f_2(a)$.

First, we need to verify that $a$ is a Dedekind cut.  Note that $a$ is clearly an initial segment of $\Q_{+}$, from its definition.  If $c\in \R_{+}$ is the maximum of the coefficients of $f_1(x)$, then for any rational (or real) number $r>\max(1,cm)$, we have
\[
f_1(r)\leq c(1+r+\cdots + r^{m-1})\leq cmr^{m-1}<r^m\leq f_2(r).
\]
Thus, $a$ is a proper subset of $\Q_+$ since it does not contain any such $r$.  (Equivalently, $f(x)$ is positive when evaluated at large values.)  A similar computation, using the reciprocal polynomial $f(0)^{-1}x^m f(1/x)$, shows that $a$ is nonempty.  Finally, if $a$ does have a greatest element, then simply remove it and now $a$ is a Dedekind cut.  (Regardless, the argument in the next paragraph will show, \emph{a posteriori}, that it did not have a greatest element after all.)

From the definition of $a$, there are rational numbers $q$ arbitrarily close to $a$ (from the left) for which $f(q)<0$.   From the continuity of polynomials, $f(a)\leq 0$.  Supposing, by way of contradiction, that $f(a)<0$, then again from the definition of $a$, there are rational numbers $q$ arbitrarily close to $a$ (from the right) for which $f(q)\geq 0$, contradicting continuity.
\end{proof}

Note that Proposition \ref{Prop:OddDegree} would also follow from the IVT.  Some have used the existence of such citations of the IVT as evidence that the FTA has no purely algebraic proof. An even stronger assertion appears in \cite{BL}:
\begin{quote}
It is a truism that the theorem is not really a theorem of algebra but of analysis or topology.
\end{quote}
Is the fundamental theorem of algebra truly innocent of the crime of being a ``theorem of algebra''?  Or is the quotation implicity presenting a false dichotomy, and the FTA is guilty of being a theorem of algebra while also being a theorem of analysis and topology?

Continuity of real polynomials is used in a central way to prove Proposition \ref{Prop:OddDegree}, but that type of continuity (and its very easy, inductive proof) immediately translates into the language of ring theory via \eqref{Eq:RingOrder}.  The careful reader might perhaps exclude the FTA from being purely algebraic because the construction of $\R_{+}$ used set-theoretic machinery, or because some important facts require the number-theoretic inductive nature of $\Z_{+}$, or because we appealed to second-order logic, but those are all standard resources.  We trust that readers will now reach a verdict of {\it guilty}.

\section{Strategy overview}\label{Section:Overview}

Let $F$ be a field, and let $K/F$ be an algebraic field extension.  When someone asserts that $K$ is an \emph{algebraic closure} of $F$, it usually means one of the following two conditions:
\begin{itemize}
\item there is no algebraic extension of $K$ (i.e., $K$ is algebraically closed) or
\item every nonconstant polynomial over $F$ splits into linear factors over $K$.
\end{itemize}
These conditions are well known to be equivalent (see \cite[p.\ 543]{DF}).

To demonstrate that $\C$ is an algebraic closure of $\R$, we will verify the second bullet point.  Thus, it suffices to check two facts: every irreducible polynomial over $\R$ is linear or quadratic and every quadratic polynomial over $\R$ has roots in $\C$.  We already verified the second fact via Lemma \ref{Lemma:QuadOverC}, and so only the first fact remains.

To finish our overview, we will describe how to handle polynomials with small degree.  There is nothing to show for degrees $1$ and $2$.  By Proposition \ref{Prop:OddDegree}, there are no irreducible polynomials of degree $3$, as expected.  It is at degree $4$ where things first become difficult.  Ostensibly, no irreducible polynomial of degree $4$ should exist.  If a degree four polynomial has no real root, it should have two irreducible quadratic factors.

This motivates a more general question: Given a polynomial $f(x)$ over a field $F$, with $\deg(f)=4$, when can we guarantee that $f(x)$ has a quadratic factor?  In the next section, we will prove that over \emph{any field} $F$, there is an extension field $K/F$ of degree $1$, $2$, $3$, $5$, or $6$ where $f(x)$ has a quadratic factor.

Consider the special case when $F=\R$.  We expect $[K:\R]=1$ because we expect that $f(x)$ already has a quadratic factor over $\R$ (even if there are no roots over $\R$).  Thus, our goal is to rule out the other four possibilities.  Proposition \ref{Prop:OddDegree} disallows the cases $[K:\R]=3,5$.  If $[K:\R]=2$, then $K=\C$ (up to isomorphism), and so $f(x)$ would have a complex quadratic factor $g(x)\in \C[x]$.  Then, by Lemma \ref{Lemma:QuadOverC}, the polynomial $g(x)$ has a root $r\in \C$.  Consequently, $f(x)\in \R[x]$ is not irreducible, because the minimal polynomial of $r$ over $\R$ divides $f(x)$.

All that remains is to rule out the case where $[K:\R]=6$.  At first blush, it appears that we have reached an impasse, for when trying to rule out the existence of an irreducible degree $4$ polynomial, we must now handle the larger degree $6$.  Fortunately, we have made some progress, because the $2$-adic valuation of $6$ is smaller than that of $4$.  Showing that there are no degree $6$ extensions of $\R$ turns out to be almost trivial.

Ultimately, we will proceed by induction not on the degree of $f(x)$, but rather on the $2$-adic valuation of the degree.

\section{Forcing factors of a given degree}\label{Section:4}

Let $F$ be a field, and let $f(x)\in F[x]$ be a nonconstant polynomial of degree $n$.  It is well known that if $f(x)$ is irreducible over $F$, then there is an extension field $K/F$ of degree $n$ where $f(x)$ has a linear factor over $K$.  If we drop the irreducibility hypothesis, we can instead take $[K:F]$ to be the degree of any irreducible factor of $f(x)$ over $F$.

What if, instead of a linear factor, we want to pass to an extension field where $f(x)$ has a quadratic factor (or a factor of a larger degree)?  We can answer this question by using mostly the same ideas as in the linear case.  To keep things general, instead of working only over fields, we work with rings.  All rings in this article are associative, unital, and commutative.

Let $R$ be a ring, and let $f(x)\in R[x]$ be a nonconstant polynomial.  We are interested in generically forcing $f(x)$ to have a monic divisor of a given fixed degree $k\geq 0$.  To that end, let $b_0,\ldots, b_{k-1}$ be independent polynomial indeterminates over $R$, and take
\[
g(x):=x^k+b_{k-1}x^{k-1}+\cdots + b_0\in (R[b_0,\ldots, b_{k-1}])[x].
\]
Applying the usual division algorithm, we have
\begin{equation}\label{Eq:Divide1}
f(x)=q(x)g(x) + r(x)
\end{equation}
for some unique quotient and remainder $q(x),r(x)\in (R[b_0,\ldots, b_{k-1}])[x]$ where
\begin{equation}\label{Eq:Divide2}
\text{either $r(x)=0$ or $\deg(r(x))<\deg(g(x))$.}
\end{equation}
This process works without any Euclidean domain assumption on $R$, because the leading coefficient of $g(x)$ is a unit.  Furthermore, if $\deg(f(x))\geq \deg(g(x))$ and if $f(x)$ is monic, then the quotient $q(x)$ is monic.

Note that $g(x)\mid f(x)$ if and only if $r(x)=0$.  Let $I$ be the ideal of $R[b_0,\ldots, b_{k-1}]$ generated by the coefficients of $r(x)$, and put
\begin{equation}\label{Eq:DefineRfk}
R_{f,k}:=R[b_0,\ldots, b_{k-1}]/I.
\end{equation}
This is the generic coefficient ring over which (the image of) $f(x)$ is forced to have a monic degree $k$ factor.  The ring $R_{f,k}$ has another useful property that we now describe.

Let $\varphi\colon R\to R'$ be a (unital) ring homomorphism from $R$ to some ring $R'$.   This gives $R'$ the structure of an $R$-algebra.  (We do not assume that $\varphi$ is injective, although this will automatically be the case if $R$ is a field and $R'$ is not the zero ring.)  The map $\varphi$ naturally extends to a ring homomorphism $R[x]\to R'[x]$ given by the rule $\sum_{i=0}^{m}r_i x^i\mapsto \sum_{i=0}^{m}\varphi(r_i)x^i$, which we will also call $\varphi$.

Write $f'(x):=\varphi(f(x))$, and assume $f'(x)$ has a monic factor $g'(x)\in R'[x]$ of degree $k$.  We then can extend $\varphi\colon R\to R'$ to a new ring homomorphism $\psi\colon R[b_0,\ldots, b_k]\to R'$ by sending $b_i$ to the $i$th coefficient of $g'(x)$ and then extending in the obvious way.  Thus, letting $\iota \colon R\to R[b_0,\ldots,b_{k-1}]$ be the natural inclusion map, then $\varphi=\psi\circ \iota$.

Just as for $\varphi$, the map $\psi$ likewise extends to polynomial rings, and we have $\psi(g(x))=g'(x)$.  Moreover, because $g'(x)\mid f'(x)$, we must have $\psi(r(x))=0$.  Thus, $\psi$ factors through the natural projection map $R[b_0,\ldots, b_{k-1}]\to R_{f,k}$, and hence so does $\varphi$.  Informally, think of $R_{f,k}$ as the universal coefficient ring where $f(x)$ is forced to have a monic factor of degree $k$.

To aid in understanding, let us consider this construction in the special case when $k=1$.  What does the universal adjunction of a monic linear factor look like?  Write
\[
f(x)=\sum_{i=0}^{n}a_ix^i\in R[x].
\]
Instead of taking $g(x)=x+b_0$, write $g(x)=x-\alpha$, as this will help simplify computations.  Using the long division process to divide $f(x)$ by $x-\alpha$ in the ring $(R[\alpha])[x]$, we find that the quotient is
\[
q(x) = a_nx^{n-1} + (a_{n-1}+a_n\alpha)x^{n-2} + \cdots + (a_1+a_2\alpha + \cdots + a_{n}\alpha^{n-1})\in (R[\alpha])[x],
\]
and that the remainder is $f(\alpha)\in R[\alpha]$.  Thus, up to the change of variables $b_0= -\alpha$, we have
\[
R_{f,1}=R[\alpha]/\langle f(\alpha)\rangle.
\]
When $R$ is a field and $f(x)$ is an irreducible polynomial, readers will recognize this ring as the usual field extension of $R$ where a root of $f(x)$ has been adjoined.  The following proposition describes an important (and well-known) structural fact regarding this ring.

\begin{prop}\label{Prop:MonicfDegree1}
Let $R$ be a nonzero ring.  If $f(x)\in R[x]$ is monic, then $R_{f,1}$ is a free  $R$-module of rank $n:=\deg(f)$.
\end{prop}
\begin{proof}
Given any element $h(\alpha)\in R[\alpha]$, then dividing by $f(\alpha)$ using long division yields a unique quotient and a unique remainder $q_h(\alpha),r_h(\alpha)\in R[\alpha]$ with
\[
h(\alpha)=q_h(\alpha)f(\alpha)+r_h(\alpha)
\]
and either $r_h(\alpha)=0$ or its degree is smaller than $n$.  Thus, in the factor ring $R_{f,1}$, cosets are uniquely represented by $R$-linear combinations of the images of $1,\alpha,\ldots, \alpha^{n-1}$.
\end{proof}

We are mainly interested in the case when $R$ is an $F$-algebra for some field $F$.  From Proposition \ref{Prop:MonicfDegree1},
\begin{equation}\label{Eq:DimensionFactorDegree1}
\dim_F(R_{f,1})=n\cdot \dim_F(R).
\end{equation}
In particular, when $R$ is finite-dimensional, then so is $R_{f,1}$.  Generalizing, we obtain the following fact about $R_{f,k}$, for any integer $k\geq 0$.

\begin{thm}\label{Thm:DimensionFactor}
Let $F$ be a field, let $R$ be a nonzero $F$-algebra, and let $f(x)\in R[x]$ be a monic polynomial of degree $n$.  For any integer $k\geq 0$,
\[
\dim_{F}(R_{f,k})=\binom{n}{k}\cdot \dim_F(R).
\]
\end{thm}
\begin{proof}
One can easily check that $R_{f,0}\cong R$, so the needed equality holds when $k=0$.  Now consider when $k\geq 1$, and inductively assume that the theorem is true for integers smaller than $k$.  Working over $R_{f,k}$, let $g(x)$ be the (image of the generic) degree $k$ monic polynomial forced to divide the image of $f(x)$.  Letting $S_1:=(R_{f,k})_{g,1}$, then \eqref{Eq:DimensionFactorDegree1} yields
\begin{equation}\label{Eq:S1}
\dim_F(S_1)=k\cdot \dim_{F}(R_{f,k}).
\end{equation}
On the other hand, working over $R_{f,k-1}$, write the image of $f(x)$ as $g'(x)h'(x)$ where $g'(x)$ is the (image of the generic) degree $k-1$ monic divisor.  Then let $S_2:=(R_{f,k-1})_{h',1}$.  By the inductive assumption together with \eqref{Eq:DimensionFactorDegree1}, noting $\deg(h'(x))=n-k+1$, we have
\begin{equation}\label{Eq:S2}
\dim_{F}(S_2)= (n-k+1)\cdot \binom{n}{k-1}\cdot \dim_{F}(R).
\end{equation}

Given a monic polynomial $f(x)$ over any nonzero ring, the following two options give exactly the same information:  (A) specifying a monic factorization $f(x)=g'(x)h'(x)$ where $\deg(g'(x))=k-1$ together with specifying a monic linear factor of $h'(x)$ and (B) specifying a monic factorization $f(x)=g(x)h(x)$ where $\deg(g(x))=k$ together with specifying a monic linear factor of $g(x)$.  Thus, from the universal natures of $S_1$ and $S_2$, there are surjective ring homomorphisms $S_1\to S_2$ and $S_2\to S_1$.  In particular, they have the same $F$-dimension.  Thus, combining \eqref{Eq:S1} with \eqref{Eq:S2} yields the claimed equality.
\end{proof}

When $k>\deg(f(x))$, then $R_{f,k}$ is the zero ring, and so the formula in Theorem \ref{Thm:DimensionFactor} still holds true.  The previous proof also continues to work, essentially vacuously.

Given an integer prime $p$, let $\nu_p$ denote the $p$-adic valuation map.  (In other words $\nu_p(n)$ is the largest integer exponent $k$ such that $p^k\mid n$.)  The final technical result needed to prove the FTA capitalizes on Theorem \ref{Thm:DimensionFactor}, allowing us to describe an interesting fact that is true for arbitrary fields.

\begin{thm}\label{Thm:2adic}
Let $F$ be a field, and let $f(x)\in F[x]\setminus\{0\}$.  If $p$ is an integer prime dividing $n:=\deg(f(x))$, then there exists an algebraic field extension $K/F$ where $f(x)$ has a factor of degree $p$ in $K[x]$, with $[K:F]\leq \binom{n}{p}$ and $\nu_p([K:F])<\nu_p(n)$.
\end{thm}
\begin{proof}
Without loss of generality, assume $f(x)$ is monic.  According to Theorem \ref{Thm:DimensionFactor}, the ring $F_{f,p}$ has $F$-dimension $\binom{n}{p}$.  As $n$ is divisible by $p$,
\[
\nu_p\left(\binom{n}{p}\right)=\nu_p\left(\frac{n(n-1)(n-2)\cdots (n-p+1)}{p!} \right) = \nu_p(n)-1 < \nu_{p}(n).
\]
This is the key inequality.

Recall that a commutative ring $A$ is \emph{local} when it has a unique maximal ideal $\mathfrak{m}$.  Now $F_{f,p}$, being finite-dimensional as an $F$-vector space, is a finite direct product of local rings.  (For several different proofs, see the answers at \cite{Stack} or see Theorem 3(4) on page 752 in \cite{DF}.)  At least one of those local rings $(A,\mathfrak{m})$ has $F$-dimension with smaller $p$-adic valuation than $n$, because the sum of all their dimensions equals $\binom{n}{p}$.  The ring $A$ has a (finite) filtration by the powers of $\mathfrak{m}$, producing the associated graded ring
\[
(A/\mathfrak{m})\oplus (\mathfrak{m}/\mathfrak{m}^2)\oplus \cdots.
\]
The $F$-dimension of $A$ equals the sum of the $F$-dimensions of these direct summands.

On the other hand, each summand of the graded ring is a module over the residue field $A/\mathfrak{m}$.  Therefore, the $F$-dimension of the residue field divides the $F$-dimension of $A$, and hence has a smaller $p$-adic valuation than $n$.  Taking $K=A/\mathfrak{m}$ we are done.
\end{proof}

Taking $n=4$ and $p=2$, we find that $[K:F]$ is limited to the values claimed in the outline given in Section \ref{Section:Overview}.

\section{Proving the fundamental theorem of algebra}\label{Section:5}

A strong version of the FTA can be expressed as follows:

\begin{thm}\label{Thm:FullFTA}
If a field $F$ satisfies the two conditions
\begin{itemize}
\item[$(1)$] $F(\sqrt{-1})$ has no quadratic extension and
\item[$(2)$] each nonlinear polynomial of odd degree is reducible in $F[x]$,
\end{itemize}
then $F(\sqrt{-1})$ is algebraically closed.
\end{thm}
\begin{proof}
Fix some algebraic closure $\overline{F}\supseteq F(\sqrt{-1})$.  Assume, by way of contradiction, that there is an irreducible polynomial $f(x)\in F[x]$ with $n:=\deg(f(x))\geq 3$.  Subject to this assumption, choose $n$ such that $\nu_2(n)$ is minimal.  Condition (2) forces $\nu_2(n)\geq 1$.

By Theorem \ref{Thm:2adic}, fix an algebraic extension $K/F$ in $\overline{F}$ with $\nu_2([K:F])<\nu_2(n)$ where $f(x)$ has a degree $2$ factor over $K[x]$.  Let $\alpha\in K$ be arbitrary, and let $m_{\alpha,F}(x)\in F[x]$ be its minimal polynomial over $F$.  Since $m_{\alpha,F}(x)$ is irreducible, but $\deg(m_{\alpha,F}(x))$ divides $[K:F]$, the minimality condition on $\nu_2(n)$ forces $m(x)$ to be linear or quadratic.  In either case, by condition (1) its roots are in $F(\sqrt{-1})$.  As $\alpha\in K$ is arbitrary, $K\subseteq F(\sqrt{-1})$.

Therefore, $f(x)$ has a degree $2$ factor in $F(\sqrt{-1})[x]$.  Then, by condition (1), it has a root $\beta\in F(\sqrt{-1})$.  Let $m_{\beta,F}(x)$ be the minimal polynomial of $\beta$ over $F$, which is at most quadratic over $F$.  Since $m_{\beta,F}(x)\mid f(x)$, this means that $f(x)$ has a proper factor in $F[x]$, contradicting the irreducibility assumption on $f(x)$.

Thus, the only irreducible polynomials over $F$ are linear or quadratic.  By condition (1), all their roots are in $F(\sqrt{-1})$, so $F(\sqrt{-1})=\overline{F}$, making it algebraically closed.
\end{proof}

The equality $\overline{\R}=\C$ is now an immediate corollary.  The conditions of Theorem \ref{Thm:FullFTA} are met when $F=\R$, by Lemma \ref{Lemma:QuadOverC} and Proposition \ref{Prop:OddDegree}.

The converse of Theorem \ref{Thm:FullFTA} is also true.  Assuming $F(\sqrt{-1})$ is algebraically closed, then condition (1) is immediate, and condition (2) follows from the fact that $F(\sqrt{-1})$ has no nontrivial odd degree extensions together with $[F(\sqrt{-1}):F]\leq 2$.

As shown by Shipman \cite{Shipman}, condition (2) of Theorem \ref{Thm:FullFTA} can be replaced by
\begin{itemize}
\item[$(2)'$] each polynomial of odd prime degree over $F$ has a root in $F$.
\end{itemize}
Using our machinery, we find that even more is true.

\begin{thm}\label{Thm:Roots}
Over any field $F$, conditions $(2)$ and $(2)'$ are equivalent.
\end{thm}
\begin{proof}
$(2)\Rightarrow (2)'$:  By induction, each polynomial of odd degree has a root.

$(2)'\Rightarrow (2)$:  Let $S$ be the set of $n\in \Z_+$ such that some polynomial over $F$ of degree $n$ has no roots in $F$.  Since the product of two polynomials without roots also has no root, the set $S$ is closed under addition.  Working by way of contradiction, assume that $S$ contains an odd integer.

Let $m\geq 1$ be the GCD of all the elements in $S$.  If $m=1$, then by the classical coin problem (also called the postage stamp problem) $S$ must contain every sufficiently large integer, which contradicts property $(2)'$.  So $m\geq 2$.  Further, since $S$ contains an odd integer, then $2\notin S$, and so we will hereafter remove the word ``odd'' from condition $(2)'$.  Let $p$ be an integer prime dividing $m$.  Fix a nonlinear irreducible polynomial $f(x)\in F[x]$ such that $\nu_p(\deg(f(x)))$ is minimal (which must be at least $1$).

By Theorem \ref{Thm:2adic}, fix an algebraic extension $K/F$ where $f(x)$ has a degree $p$ factor in $K[x]$, with $\nu_{p}([K:F])<\nu_p(\deg(f(x)))$.  If $K=F$, then $f(x)$ has a degree $p$ factor in $F[x]$.  Hence, it has a root in $F$, contradicting the defining condition for $f(x)$.  Thus, $K\neq F$, and we may fix some $\alpha\in K\setminus F$.  Letting $m_{\alpha,F}(x)$ be its minimal polynomial, then $m_{\alpha,F}(x)$ has no root in $F$.  However, $\deg(m_{\alpha,F}(x))$ divides $[K:F]$, contradicting the $p$-adic minimality assumption on $\deg(f(x))$.
\end{proof}

\section{Final remarks and optimizations}

In this final section, we suggest possible improvements to Theorem \ref{Thm:FullFTA}, especially when we allow ourselves access to deeper theory.  These improvements come in three forms: simplifying the proof, strengthening the theorem's conclusion, and weakening (or optimizing in other ways) the assumed hypotheses (1) and (2).

To begin, we note that in the usual case when $F=\R$, Theorems \ref{Thm:DimensionFactor} and \ref{Thm:2adic} are only needed when $k=p=2$ and when $R=F$.  The proofs of these two theorems can be somewhat simplified in this situation.  On the other hand, by utilizing some non-elementary ring theory, Theorem \ref{Thm:DimensionFactor} may be improved to show that $R_{f,k}$ is a free $R$-module of rank $\binom{n}{k}$.  We do not include details as these improvements are orthogonal to the purposes of this paper.

An interesting aspect of our work is that, outside of Section \ref{Section:2} and this present section, we make no use of the characteristic of fields.  It appears that all previous generalizations of the FTA to arbitrary characteristic utilize facts about inseparability.  It is a pleasant surprise that our proofs of Theorems \ref{Thm:FullFTA} and \ref{Thm:Roots} are entirely characteristic-free.  Further, the dichotomy inherent in the cases $F=F(\sqrt{-1})$ and $F\neq F(\sqrt{-1})$ is entirely irrelevant.

A classic theorem of Artin and Schreier from 1927 says that $\overline{F}/F$ is a nontrivial finite extension if and only if $F$ is real closed.   Appealing to that theorem, we may improve Theorem \ref{Thm:FullFTA}, by generalizing its conclusion, as follows:

\begin{thm}
A field $F$ satisfies conditions $(1)$ and $(2)$ if and only if $\overline{F}/F$ is a finite extension, in which case $\overline{F}=F(\sqrt{-1})$.
\end{thm}

As it will be useful in the following discussion, we record here an important part of the theory built up by Artin and Schreier, restricted to degree $2$ extensions.

\begin{lemma}\label{Lemma:KeyLemma}
For any field $F$ containing $\sqrt{-1}$, if $F$ has a quadratic extension, then it has extensions of every power-of-$2$ degree.
\end{lemma}
\begin{proof}
When the extension is inseparable, this is immediate.  When the extension is separable and ${\rm char}(F)=2$, this follows from Lemma 3 on page 675 in \cite{Jacobson}.  Finally, when ${\rm char}(F)\neq 2$, the quadratic extension is of the form $F(\sqrt{a})$ for some $a\in F\setminus F^2$.  Then  for each $r\in \Z_+$ the polynomial $x^{2^r}-a$ is irreducible in $F[x]$ by \cite[Theorem 9.1]{Lang}.
\end{proof}

Recall that Theorem \ref{Thm:Roots} gives us the freedom to express the FTA either in terms of the existence of roots or in terms of reducibility.  Shipman thoroughly investigated the ``roots'' case in \cite{Shipman}.  We have nothing to add to that discussion, except to note that it would be an interesting project to generalize Shipman's results using Theorem \ref{Thm:2adic} if possible (since roots correspond to degree $1$ factors).

We now turn to optimizing assumptions (1) and (2) in terms of reducibility hypotheses.  To that end, for each integer $n\geq 2$, consider the condition
\[
C(n):\ \text{every polynomial of degree $n$ (over the given field) is reducible}.
\]
This is generally a much weaker hypothesis than assuming every polynomial of degree $n$ has a root (except when $n=2,3$), thus potentially leading to a stronger version of the FTA.  The condition $C(n)$ is equivalent to saying that the given field has no extension of degree $n$, by \cite[Theorem 1]{GS}.  (There is a minor oversight in the inseparable case of the proof, where $\theta$ may need to be replaced by $c\theta$, for some $c\in k\setminus k^p$.)  Hereafter, we will primarily think of $C(n)$ as a statement about the nonexistence of a degree $n$ extension.

The following proposition gives two possible refinements of condition $(1)$. Interestingly, the first condition does not isolate $F(\sqrt{-1})$ for special consideration.

\begin{prop}\label{Prop:1Equiv}
Let $F$ be a field.  Consider the properties
\begin{enumerate}
\item[$({\rm A}1)$] $F$ satisfies $C(2^r)$ for some $r\in \Z_{+}$ and
\item[$({\rm B}1)$] $F(\sqrt{-1})$ satisfies $C(2^r)$ for some $r\in \Z_{+}$.
\end{enumerate}
The following implications hold
\[
({\rm A}1)\Longrightarrow ({\rm B}1) \Longleftrightarrow (1).
\]
\end{prop}
\begin{proof}
$({\rm A}1)\Rightarrow ({\rm B}1)$:  Assume $C(2^r)$ holds over $F$, for some $r\in \Z_+$.  If $F=F(\sqrt{-1})$, then $C(2^r)$ also holds over $F(\sqrt{-1})$.  If $F\neq F(\sqrt{-1})$, then $r>1$ and $C(2^{r-1})$ holds over $F(\sqrt{-1})$.

$({\rm B}1)\Leftrightarrow (1)$:   If $(1)$ holds, then $({\rm B}1)$ holds for tautological reasons.  Assume $(1)$ fails, so $F(\sqrt{-1})$ has a quadratic extension.  By Lemma \ref{Lemma:KeyLemma}, $({\rm B}1)$ also fails.
\end{proof}

We should mention that the conditions of Proposition \ref{Prop:1Equiv} can be interpreted Galois-theoretically.  We leave that task to the interested reader, as well as the task of showing that the field of real constructible numbers (i.e., the subfield of $\R$ obtained by closing $\Q$ under square-roots of positive numbers) satisfies $({\rm B}1)$ but not $({\rm A}1)$.  While condition $({\rm A}1)$ is thus in general strictly stronger than $(1)$, it is nevertheless a necessary condition for $\overline{F}=F(\sqrt{-1})$ to hold, hence making it a natural replacement for $(1)$ if one wants to avoid mentioning the field $F(\sqrt{-1})$ specifically.  Moreover, there exists a field whose extension degrees are exactly the powers of $2$, and so $({\rm A}1)$ is a required assumption when stating the FTA only in terms of possible extension degrees for the \emph{base} field $F$.

We now turn to condition $(2)$, which says precisely that $C(n)$ holds over $F$ for each odd integer $n>1$.  This statement is not optimal, in the sense that only some of the odd numbers are needed.  Indeed, $(2)$ is equivalent to:
\begin{itemize}
\item[$({\rm A}2)$] $F$ satisfies $C(n)$ for all sufficiently large odd integers $n\gg 1$.
\end{itemize}
\begin{proof}[Proof sketch]
Assume there is a nontrivial odd degree extension $K/F$.  If it is inseparable, there are arbitrarily large such extensions.  In the other case, instead use \cite[Theorem 7]{GS} to get extensions with arbitrarily large odd part.  Then pass to a Galois closure, and then to the fixed field of a Sylow $2$-subgroup.
\end{proof}

This leaves open the possibility that the set of odd numbers used in condition $(2)$ could be optimized in other ways.  For example, perhaps we could assume $C(n)$ holds only when $n$ is an odd prime power.  However, such ruminations are complicated by the following:

\begin{example}\label{Example:A6}
Let $F_0$ be a field with some $\alpha\in \overline{F_0}$ such that $F_0(\alpha)/F_0$ is Galois with Galois group $A_6$.  (Such a field exists.)  Take $F$ be a maximal subfield of $\overline{F_0}$ such that $F(\alpha)/F$ continues to have Galois group $A_6$.  Given any finite, nontrivial extension $K/F$ in $\overline{F_0}$, then $K(\alpha)/K$ is not an $A_6$ extension from the maximality hypothesis on $F$.  Hence, $K\cap F(\alpha)\neq F$ by appealing to \cite[Proposition 19, p.\ 591]{DF}.  Therefore, $[K:F]$ is a multiple of $[L:F]$ for some field $L$ with $F\subsetneq L\subseteq F(\alpha)$.  We may as well take $L$ to be a minimal nontrivial extension of $F$ in $F(\alpha)$.  In that case $[L:F]$ is the index of a maximal subgroup of $A_6$.  The only such indexes are $6$, $10$, and $15$.  Consequently, any set of odd numbers used to limit condition (2) must contain a multiple of $15$.
\end{example}

In the previous example, $A_6$ can be replaced by any finite simple group.  Disappointingly, this means that to understand exactly which combinations of the $C(n)$'s are sufficient to guarantee $\overline{F}=F(\sqrt{-1})$, one must (at the very least) know the indexes of the maximal subgroups of such groups.  We might appeal to the classification of finite simple groups to start collecting such information, but then we are far removed from \emph{fundamental} mathematics.  Generalizing the FTA along these lines might make it a misnomer after all.

\bigskip

\noindent{\bf Author Contribution}: Katelyn S.\ Clark and Pace P.\ Nielsen contributed equally in all aspects of creating this material.
\bigskip

\noindent{\bf Funding}: This work was partially supported by a grant from the Simons Foundation (\#963435 to Pace P.\ Nielsen).
\bigskip

\noindent{\bf Acknowledgements}: We thank both Kyle Pratt for extensive comments on a previous draft and Alexander Lai De Oliveira for useful conversations on Theorem \ref{Thm:DimensionFactor}.  We also thank multiple anonymous reviewers for suggestions that improved this paper.

\providecommand{\MR}{\relax\ifhmode\unskip\space\fi MR }
\providecommand{\MRhref}[2]{%
  \href{http://www.ams.org/mathscinet-getitem?mr=#1}{#2}
}
\providecommand{\href}[2]{#2}

\end{document}